\documentclass[12pt,reqno]{amsart}
\usepackage{amsfonts}
\usepackage{amsmath}
\usepackage{amsthm}
\usepackage{mathrsfs}
\usepackage{stmaryrd}
\usepackage{graphicx}
\usepackage{epstopdf}
\usepackage{textcomp}

\newtheorem{theorem}{Theorem}
\newtheorem{lemma}{Lemma}
\newcommand{\Tr}{{\rm Tr\,}}
\theoremstyle{definition}

\newtheorem{remark}{Remark}
\begin{document}
\title{Central limit theorem under uncertain linear transformations}

\author{Dmitry B. Rokhlin}

\address{Institute of Mathematics, Mechanics and Computer Sciences,
              Southern Federal University,
Mil'chakova str., 8a, 344090, Rostov-on-Don, Russia}
\email{rokhlin@math.rsu.ru}
\thanks{The research is supported by Southern Federal University, project 213.01-07-2014/07.}

\begin{abstract}
We prove a variant of the central limit theorem (CLT) for a sequence of i.i.d. random variables $\xi_j$, perturbed by a stochastic sequence of linear transformations $A_j$,
representing the model uncertainty. The limit, corresponding to a "worst" sequence $A_j$, is expressed in terms of the viscosity solution of the $G$-heat equation. In the context of the CLT under sublinear expectations this nonlinear parabolic equation appeared previously in the papers of S.\,Peng. Our proof is based on the technique of half-relaxed limits from the theory of approximation schemes for fully nonlinear partial differential equations.
\end{abstract}

\maketitle

\section{Problem formulation} \label{sec:1} 
\setcounter{equation}{0}

Consider a sequence of i.i.d. $d$-dimensional random variables $(\xi_j)_{j=1}^\infty$, $\xi_j=(\xi_j^r)_{r=1}^d$. Denote by $\xi$ a random variable distributed as $\xi_j$, and assume that
\begin{equation} \label{eq:1.1}
\mathsf E\xi=0, \quad \left(\mathsf E(\xi^r\xi^l)\right)_{r,l=1}^d=I,
\end{equation}
where $I$ is the identity matrix. By the classical central limit theorem (CLT), for any bounded continuous function $f:\mathbb R^d\mapsto\mathbb R$  and a fixed  $d\times d$ matrix $A$,  we have
$$ \lim_{n\to\infty}\mathsf E f\left(\sum_{j=1}^n \frac{A\xi_j}{\sqrt n}\right)=\mathsf E f(A\eta),$$
where $\eta$ has the standard $d$-dimensional normal law. Note, that for given $f$ the limit depends only on the covariance matrix $AA^T$ of $A\xi$.

In this paper we consider the case where $A$ is not known exactly, and can change dynamically within a prescribed set. This is a simple example of a probability model under uncertainty. The extension of the CLT, obtained below, looks similar to Peng's CLT under sublinear expectations: \cite{Peng07,Peng08}. However, our problem formulation, as well as the proof, do not involve the nonlinear expectations theory in any way. On the other hand, similarly to Peng's approach, the key role is played by the viscosity solutions theory.

Consider a filtered probability space $(\Omega,\mathscr F,\mathbb P, (\mathscr F_j)_{j=0}^\infty)$ and an adapted sequence $(\xi_j)_{j=1}^\infty$ of $d$-dimensional random variables such that $\xi_j$ is independent from $\mathscr F_{j-1}$ and satisfy (\ref{eq:1.1}).
Denote by $\mathbb M_d$ (resp., $\mathbb S_d$) the set of $d\times d$ matrices (resp., symmetric matrices).
Let $(A_j)_{j=0}^\infty$ be an adapted sequence with values in a compact set $\Lambda\in\mathbb M_d$. The process $(A_j,\mathscr F_j)_{j=0}^\infty$, where  is $\mathscr F_j$, $j\ge 1$ may be wider than $\sigma(\xi_1,\dots,\xi_j)$, is chosen by the ``nature", and represents the ``Knightian uncertainty".

Our goal is to describe the limit
\begin{equation} \label{eq:1.2}
\mathscr L:=\lim_{n\to\infty}\sup_{A_0^{n-1}\in\mathfrak A_0^{n-1}}\mathsf E f\left(\sum_{j=0}^{n-1} \frac{A_j\xi_{j+1}}{\sqrt n}\right),
\end{equation}
where $f$ is a bounded continuous function on $\mathbb R^d$, $A_s^{n-1}=(A_j)_{j=s}^{n-1}$ and $\mathfrak A_s^{n-1}=\{A_s^{n-1}| A_j:\Omega\mapsto\Lambda \text{ and } A_j \text{ is } \mathscr F_j\text{-measurable}\}$. Besides the theoretical interest, such quantities are useful, e.g., for measuring risk and option pricing under uncertain volatility: see \cite{AveLevPar95}.

Our result (Theorem \ref{th:1}) and its discussion are presented in the next section. The proof is deferred to Section \ref{sec:3}. In Section \ref{sec:4} we discuss the relationship of our problem to the nonlinear expectations theory.

\section{Central limit theorem under uncertain linear transformations}
\label{sec:2}
\setcounter{equation}{0}
First we formulate our result and then give some comments.

\begin{theorem} \label{th:1}
Let $v:[0,1]\times\mathbb R^d\mapsto\mathbb R$ be the unique continuous viscosity solution of the nonlinear parabolic equation
\begin{equation} \label{eq:2.1}
 -v_t(t,x)-\frac{1}{2}\sup_{A\in\Lambda}\Tr\left(A A^T v_{xx}(t,x)\right)=0, \quad (t,x)\in [0,1)\times \mathbb R^d,
\end{equation}
satisfying the terminal condition
\begin{equation} \label{eq:2.2}
v(1,x)=f(x),\quad x\in\mathbb R^d.
\end{equation}
Then $\mathscr L=v(0,0)$.
\end{theorem}

The equation (\ref{eq:2.1}) can be written in the form
$$ -v_t-G(v_{xx})=0,\quad\text{where } G(S)=\frac{1}{2}\sup_{A\in\Lambda}\Tr\left(A A^T S\right),\ S\in \mathbb S_d.$$
By $v_{xx}=(v_{x^r x^l})_{r,l=1}^d$ we denote the Hessian matrix.

Let us recall the definitions of viscosity semisolutions: see, e.g., \cite{CraIshLio92}. Put $Q^\circ=[0,1)\times\mathbb R^d$, $Q=[0,1]\times\mathbb R^d$ and denote by $C^2_b(\mathbb R^{d+1})$ the set of test functions, whose derivatives up to the second order, are continuous and bounded. A bounded upper semicontinous (usc) (resp., lower semicontinuous (lsc)) function $u:Q\mapsto\mathbb R$,  is called a \emph{viscosity subsolution} (resp., \emph{supersolution}) of the problem (\ref{eq:2.1}), (\ref{eq:2.2}) if $u(1,x)\le f(x)$ (resp., $u(1,x)\ge f(x)$) on $\mathbb R^d$ and for any $\varphi\in C^2_b(\mathbb R^{d+1})$, $(\overline t,\overline x)\in Q^\circ$ such that $(\overline t,\overline x)$ is the strict global maximum (resp., minimum) point of $u-\varphi$ on $Q^\circ$, the inequality
$$ -\varphi_t(\overline t,\overline x)-G\left(\varphi_{xx}(\overline t,\overline x)\right)\le 0\quad (\text{resp.}, \ge 0)$$
holds true. As is well known, in these definitions one can equivalently assume that the extremum is local or not strict.

A bounded \emph{continuous} function $u:Q\mapsto\mathbb R$ is called a \emph{viscosity solution} of (\ref{eq:2.1}), (\ref{eq:2.2}), if it is a viscosity sub- and supersolution.

\begin{remark} \label{rem:1}
The uniqueness of a bounded continuous viscosity solution of  (\ref{eq:2.1}), (\ref{eq:2.2}) is well known: see, e.g., \cite{Nis15} (Theorem 3.5). However, to prove Theorem \ref{th:1} we need a more subtle result. Note, that we require the equation (\ref{eq:2.1}) to be satisfied in the viscosity sense at the lower boundary of $Q$. So, by the accessibility theorem of \cite{CheGigGot91},
for an usc viscosity subsolution $u$ and a lsc viscosity supersolution $w$ of (\ref{eq:2.1}) we have
$$ u(0,x)=\limsup_{\substack{(t,y)\in (0,1)\times\mathbb R^d,\\ t\to 0, y\to x}} u(t,y);\quad
   w(0,x)=\liminf_{\substack{(t,y)\in (0,1)\times\mathbb R^d,\\ t\to 0, y\to x}} w(t,y).
$$
Using this fact we may apply the comparison result of \cite{DieFriObe14} (Theorem 1) and conclude that $u\le w$ on $Q$.
\end{remark}

\begin{remark} \label{rem:2}
Consider the system of stochastic differential equations
\begin{equation} \label{eq:2.3}
dY_s=A_s dW_s,\quad s\in [t,1],\quad Y_t=x,
\end{equation}
where $W$ is a standard $d$-dimensional Brownian motion. Denote by $\mathcal A_t$ the set of stochastic process $(A_s)_{s\in [t,1]}$, which are progressively measurable with respect to the minimal augmented filtration, generated by $W$ (see Chap. 1 of \cite{Bass11}), and take values in $\Lambda$. From the theory of stochastic optimal control (see \cite{Tou13,Nis15}) we know that the value function
\begin{equation} \label{eq:2.4}
V(t,x)=\sup_{A\in\mathcal A_t}\mathsf E f(Y_1)
\end{equation}
is a continuous viscosity solution of the Hamilton-Jacobi-Bellman equation (\ref{eq:2.1}). Taking $(t,x)=(0,0)$, from Theorem \ref{th:1}, we obtain the stochastic control representation of the limit (\ref{eq:1.2}):
$$ \mathscr L=V(0,0)=\sup_{A\in\mathcal A_0}\mathsf E f\left(\int_0^1 A_s\, dW_s\right).$$
\end{remark}

\begin{remark} \label{rem:3}
The Euler scheme
$$ Y_{j/n+1/n}=Y_{j/n}+ A_{j/n} (W_{j/n+1/n}-W_{j/n}),\quad j=0,\dots,n-1$$
for (\ref{eq:2.3}) gives the approximation
$$ V(0,0)\approx \sup_{A_{j/n}\in\Lambda}\mathsf E f\left(\sum_{j=0}^{n-1} \frac{A_{j/n}\eta_{j+1}}{\sqrt n}\right),\quad \eta_{j+1}=\sqrt n (W_{j/n+1/n}-W_{j/n})$$
for the value function (\ref{eq:2.4}), similar to the expression in (\ref{eq:1.2}). However, $\eta_j$, in contrast to $\xi_j$, are normal. From Theorem \ref{th:1} it follows that  to get a correct approximation of $(\ref{eq:2.4})$, one can take in the Euler scheme instead of $\eta_j$ any i.i.d. random vectors $\xi_j$, satisfying (\ref{eq:1.1}).
\end{remark}

\begin{remark} \label{rem:4}
Equation (\ref{eq:2.1}) is called $G$-heat equation: \cite{Peng07b,Peng08b}. It was used by S.\,Peng for the description of the $G$-normal distribution. Moreover, the representation of the same form, as given in Theorem \ref{th:1}, appeared in the CLT under sublinear expectations: see \cite{Peng07,Peng08}.
\end{remark}

\begin{remark} \label{rem:5}
Similar to the classical case, for fixed $f$, the limit $\mathscr L$ depends only on the set $\{A A^T: A\in\Lambda\}$ of possible covariance matrices  of $A\xi$. We stress that $\mathscr L$ does not depend on the choice of filtration $\mathscr F_j\supset\sigma(\xi_1,\dots,\xi_j)$.
\end{remark}

To prove Theorem \ref{th:1} let us introduce the state variables $X_j$ by
\begin{equation} \label{eq:2.5}
 X_{j+1}=X_j+\frac{A_j\xi_{j+1}}{\sqrt n},\quad X_s=x,\quad j=s,\dots,n-1.
\end{equation}
Denote the solution of (\ref{eq:2.5}) by $X^{s,x,A}$ and consider the value functions
\begin{equation} \label{eq:2.6}
v_n(1,x)=f(x),
\end{equation}
\begin{equation} \label{eq:2.7}
v_n(s/n,x)=\sup_{A_s^{n-1}\in\mathfrak A_s^{n-1}}\mathsf E f\left(X_n^{s,x,A}\right),\quad s=0,\dots, n-1.
\end{equation}
Clearly,
\begin{equation} \label{eq:2.8}
\mathscr L=\lim_{n\to\infty}\sup_{A_0^{n-1}\in\mathfrak A_0^{n-1}}\mathsf E f\left(X_n^{0,0,A}\right)=\lim_{n\to\infty} v_n(0,0).
\end{equation}

Our goal is to prove that $v_n(0,0)\to v(0,0)$, where $v$ defined in Theorem \ref{th:1}. We apply the half-relaxed limits technique of \cite{BarPer88}, which became standard in the theory of approximation schemes for fully nonlinear second order elliptic and parabolic equations after the seminal paper \cite{BarSou91}. Namely, we construct the half-relaxed limits $\underline v\le\overline v$ of $v_n$, and prove that they are viscosity semisolutions of (\ref{eq:2.1}), (\ref{eq:2.2}). Then we use the comparison result, mentioned in Remark \ref{rem:1}, to prove the opposite inequality: $\underline v\ge\overline v$. It follows that $v=\underline v=\overline v$ and $v_n\to v$.

Note, that the strategy of \cite{Peng07,Peng08}
is different and based on the following fact:
\begin{equation} \label{eq:2.9}
\lim_{n\to\infty}\sup_{A_0^{n-1}\in\mathfrak A_0^{n-1}}\mathsf Ev(1,X_n^{0,0,A})=v(0,0).
\end{equation}
The direct proof of (\ref{eq:2.9}) requires the interior H\"{o}lder regularity of $v$, which is guaranteed in the strong parabolic case, and Lipshitz continuity of $f$. The general case requires perturbations and approximations: see \cite{Peng08}.

\section{The proof of Theorem \ref{th:1}}
\label{sec:3}
\setcounter{equation}{0}
The discrete time stochastic control problem (\ref{eq:2.5}), (\ref{eq:2.6}), (\ref{eq:2.7}) looks quite standard. However, to model Knightian uncertainty,
we consider the class $\mathcal A_n$ of ``open-loop" strategies $A_j$, adapted to an arbitrary filtration $\mathscr F_j\supset\sigma(\xi_1,\dots,\xi_j)$ (cf. \cite{Sir14}). So, we prefer to give a direct proof of the dynamic programming principle, instead of trying to find an appropriate reference.
\begin{lemma} \label{lem:1}
Put $t_j^n=j/n, j=0,\dots,n.$
The value functions $v_n$ are continuous in the state variable $x$ and uniformly bounded:
\begin{equation} \label{eq:3.1}
 |v_n(t_j^n,x)| \le\sup_{x\in\mathbb R^d} |f(x)|.
\end{equation}
Moreover, they satisfy the recurrence relations
\begin{equation} \label{eq:3.2}
v_n(1,x)=f(x),
\end{equation}
\begin{equation} \label{eq:3.3}
 v_n(t_j^n,x)=\sup_{A\in\Lambda} \mathsf E v_n(t_{j+1}^n,x+A\xi/\sqrt n),\quad j=0,\dots,n-1.
\end{equation}
\end{lemma}
\begin{proof} Consider the sequence $v_n$, defined by (\ref{eq:3.2}), (\ref{eq:3.3}). Clearly, $v_n(1,\cdot)$ is continuous and bounded.
Assume that $v_n(t^n_{j+1},\cdot)$ has the same properties. Then
$ F^n_j(x,A)=\mathsf E v_n(t^n_{j+1},x+A\xi/\sqrt n)$
is continuous in $(x,A)$ by the dominated convergence theorem, and the function
$$ v_n(t_j^n,x)=\sup_{A\in\Lambda} F^n_j(x,A)$$
is bounded and continuous in $x$ by the compactness of $\Lambda$.

Furthermore, by the Dubins-Savage type measurable selection result \cite{Schal74}, \cite{Sri98} (Theorem 5.3.1), there exists a Borel measurable function $\Phi_j:\mathbb R^d\mapsto\Lambda$ such that
\begin{equation} \label{eq:3.4}
v_n(t_j^n,x)=F^n_j(x,\Phi_j(x))=\mathsf E v_n(t_{j+1}^n,x+\Phi_j(x)\xi/\sqrt n),\quad x\in\mathbb R^d.
\end{equation}
Define the process $X$ by the recurrence relations
\begin{equation} \label{eq:3.5}
X_{j+1}=X_j+\Phi_j(X_j)\xi_{j+1}/\sqrt n,\quad X_s=x,
\end{equation}
and put $A^*_j=\Phi_j(X_j)$. The process (\ref{eq:3.5}) can be written as $X^{s,x,A^*}$. From (\ref{eq:3.4}) we get
$$ v_n(t_j^n,X^{s,x,A^*}_j)=\int_{\mathbb R^d} v_n(t_{j+1}^n,X^{s,x,A^*}_j+ \Phi_j(X_j^{s,x,A^*}) z/\sqrt n)\,d\mathsf P_\xi(dz),$$
where $\mathsf P_\xi$ is the distribution of $\xi$. Taking the expectation and using the independence of $X_j^{s,x,A^*}$ and $\xi_{j+1}$, we obtain
$$ \mathsf E v_n(t_j^n,X^{s,x,A^*}_j)=\mathsf E v_n(t_{j+1}^n,X^{s,x,A^*}_j+ \Phi_j(X_j^{s,x,A^*}) \xi_{j+1}/\sqrt n)=
   \mathsf E v_n(t_{j+1}^n,X^{s,x,A^*}_{j+1}).$$
Thus,
\begin{equation} \label{eq:3.6}
 v_n(t_s^n,x)=\mathsf E v_n(t_s^n,X^{s,x,A^*}_s)=\mathsf E v_n(t_n^n,X^{s,x,A^*}_n)=\mathsf E f(X^{s,x,A^*}_n).
\end{equation}

On the other hand, $v_n(t_j^n,x)\ge\mathsf E v_n(t_{j+1}^n,x+A\xi/\sqrt n)$, $A\in\Lambda$,
and for any sequence $(A_j)_{j=s}^{n-1}\in\mathfrak A_s^{n-1}$ we have
$$ v_n(t_j^n,X_j^{s,x,A})\ge \int_{\mathbb R^d} v_n(t_{j+1}^n,X_j^{s,x,A}+A_j z/\sqrt n)\,d\mathsf P_\xi(dz).$$
Taking the expectation, and using the independence of $(X_j^{s,x,A},A_j)$ and $\xi_{j+1}$, we obtain
$$ \mathsf E v_n(t_j^n,X^{s,x,A}_j)\ge\mathsf E v_n(t_{j+1}^n,X^{s,x,A}_j+ A_j \xi_{j+1}/\sqrt n) =\mathsf E v_n(t_{j+1}^n,X^{s,x,A}_{j+1}).$$

Hence,
\begin{equation} \label{eq:3.7}
v_n(t_s^n,x)=\mathsf E v_n(t_s^n,X^{s,x,A}_s)\ge \mathsf E v_n(t_n^n,X^{s,x,A}_n)=\mathsf E f(X^{s,x,A}_n)
\end{equation}
for any $(A_j)_{j=s}^{n-1}\in\mathfrak A_s^{n-1}$. Combining (\ref{eq:3.6}) and (\ref{eq:3.7}), we conclude that the function, defined by the recurrence relations (\ref{eq:3.2}), (\ref{eq:3.3}), is the same as the function (\ref{eq:2.6}), (\ref{eq:2.7}).
The inequality (\ref{eq:3.1}) follows from (\ref{eq:2.7}).
\end{proof}

As follows from the proof, an ``optimal strategy'' $A_j^*$ of the nature uses only the information on the current state $X_j$, although the available information $\mathscr F_j$ may be much richer (see \cite{Sir14} for a similar conclusion).

Consider a closed set $U\in\mathbb R^m$ and a sequence $U_n\subset U$ of its closed subsets such that for any $x\in U$ there exist exists a sequence $x_k\in U_{n_k}$, $n_k\in\mathbb N$ converging to $x$. For a uniformly bounded sequence of continuous functions $u_n:U_n\mapsto\mathbb R$, $|u_n|\le M$ define Barles-Perthame type half-relaxed limits $\underline u, \overline u:U\mapsto\mathbb R$ as follows
$$ \underline u(x)=\inf\{\lim u_{n_k}(x_k): x_k\in U_{n_k}, x_k\to x \text{ and } u_{n_k}(x_k) \text{ converges}\};$$
$$ \overline u(x)=\sup\{\lim u_{n_k}(x_k): x_k\in U_{n_k}, x_k\to x \text{ and } u_{n_k}(x_k) \text{ converges}\}.$$

It follows from the definitions that there exist sequences $n_k\in\mathbb N$, $x_k\in U_{n_k}$ such that $x_k\to x$, $u_{n_k}(x_k)\to\underline u(x)$ (resp., $u_{n_k}(x_k)\to\overline u(x)$).

The proofs of the next two lemmas follow the argumentation of \cite{BarCap97} (Chap. V, Lemmas 1.5, 1.6).

\begin{lemma} \label{lem:2}
$\underline u$ is lsc, $\overline u$ is usc.
\end{lemma}
\begin{proof}
If $\underline u$ is not lsc at $x\in U$, then there exist $\delta>0$, $J\in \mathbb N$ and a sequence $y_j\in U$, $y_j\to x$ such that $\underline u(y_j)\le \underline u(x)-\delta$ for $j\ge J$. By the definition of $\underline u$ for each $j$ there exist $n_j\in\mathbb N$, $x_j\in U_{n_j}$, $|x_j-y_j|<1/j$ such that
$$ u_{n_j}(x_j)\le \underline u(y_j)+\delta/2.$$
Thus,
$ u_{n_j}(x_j)\le \underline u(x)-\delta/2$, $j\ge J$ in contradiction with the definition of $\underline u$.

The case of $\overline u$ is considered in the same way.
\end{proof}

\begin{lemma} \label{lem:3}
Let $\varphi\in C_b^2(\mathbb R^m)$, $\overline x\in U$. If $\overline x$ is the strict global minimum (resp., maximum) point of $\underline u-\varphi$ (resp., $\overline u-\varphi$) on $U$, then there exist sequences $n_k\in\mathbb N$, $y_k\in U_{n_k}$ such that $y_k\to x$, $u_{n_k}(y_k)\to\underline u(\overline x)$ (resp., $u_{n_k}(y_k)\to\overline u(\overline x)$), and a test function $\psi$ such that $\psi_x(\overline x)=\varphi_x(\overline x)$, $\psi_{xx}(\overline x)=\varphi_{xx}(\overline x)$ and $y_k$ is a global minimum (resp., maximum) point of $u_{n_k}-\psi$ on $U_{n_k}$.
\end{lemma}
\begin{proof}
We consider the case of $\underline u$.
Let a sequence $x_k\to \overline x$, $x_k\in U_{n_k}$ be such that $u_{n_k}(x_{n_k})\to\underline u(\overline x)$, and let $y_k$ be a minimum point of $u_{n_k}-\varphi$ on the set $U_{n_k}\cap B_1$, $B_1=\{y:|y-\overline x|\le 1\}$. Then
\begin{equation} \label{eq:3.8}
(u_{n_k}-\varphi)(y_k)\le (u_{n_k}-\varphi)(x_k)
\end{equation}
for sufficiently large $k$. Passing, if necessary, to a subsequence, we may assume that $y_k\to \overline y\in U$, $u_{n_k}(y_k)\to z$. Thus, by (\ref{eq:3.8}) and the definition of $\underline u$, we get
$$ \underline u(\overline y)-\varphi(\overline y)\le z-\varphi(\overline y)\le \underline u(\overline x)-\varphi(\overline x).$$

It follows that $\overline y=\overline x$, $z=\underline u(\overline x)$, since $\overline x$ is the strict minimum point of $\underline u-\varphi$ on $U$. Furthermore, $y_k\to\overline x$ is a local minimum point of $u_{n_k}-\varphi$ on $U_{n_k}$, since $y_k$ lies inside the ball $B_1$ for sufficiently large $k$.

Let $\chi\in C_b^2$ be a function such that $\chi(x)=0$, $|x-\overline x|\le 1/2$, $\chi(x)=1$, $|x-\overline x|\ge 1$. Then there exist $M'>0$ such that $y_k$ is a global minimum point of $u_{n_k}(x)-\varphi(x)+M'\chi(x)$ on $U_{n_k}$ for $k$ large enough. The sequence $y_{n_k}$ and the test function $\psi=\varphi-M'\chi$ have the desired properties.
\end{proof}

In the theory of approximation schemes the following relation is known as a consistency condition: see \cite{BarSou91}.

\begin{lemma} \label{lem:4}
Let $\varphi\in C_b^2(\mathbb R\times\mathbb R^d)$ and $(t_n,x_n)\to (\overline t,\overline x)$. Then
$$\lim_{n\to\infty} n\sup_{A\in\Lambda}\mathsf E\left(\varphi(t_n+1/n,x_n+A\xi/\sqrt n)-\varphi(t_n,x_n)\right)=
\varphi_t(\overline t,\overline x)+G\left(\varphi_{xx}(\overline t,\overline x)\right).$$
\end{lemma}
\begin{proof}
By Taylor's formula we get
\begin{align} \label{eq:3.9}
& \varphi(t_n+1/n,x_n+A\xi/\sqrt n)-\varphi(t_n,x_n)=\varphi(t_n+1/n,x_n+A\xi/\sqrt n)\nonumber\\
- &\varphi(t_n,x_n+A\xi/\sqrt n)+ \varphi(t_n,x_n+A\xi/\sqrt n)-\varphi(t_n,x_n)\nonumber\\
= & \frac{1}{n}\varphi_t(\widehat t_n,x_n+A\xi/\sqrt n)+\frac{1}{\sqrt n}\varphi_x(t_n,x_n) A\xi
+  \frac{1}{2n}\varphi_{xx}(t_n,\widehat x_n)A\xi\cdot A\xi,
\end{align}
where
$\widehat t_n=t_n+\alpha_n/n$, $\widehat x_n=x_n+\beta_n A\xi/\sqrt n$, $\alpha_n,\beta_n\in [0,1].$
Note, that the sum of the first and third terms in the last line of (\ref{eq:3.9}) is $\sigma(\xi)$-measurable, and the expectation of the second term is $0$ since $\mathsf E\xi=0$.

Using (\ref{eq:3.9}), we get the inequality
\begin{gather*}
 \left|n\sup_{A\in\Lambda}\mathsf E\left(\varphi(t_n+1/n,x_n+A\xi/\sqrt n)-\varphi(t_n,x_n)\right)-\left(\varphi_t(\overline t,\overline x)+G\left(\varphi_{xx}(\overline t,\overline x)\right)\right)\right|\\
 \le \sup_{A\in\Lambda}\left|\mathsf E\varphi_t(\widehat t_n,x_n+A\xi/\sqrt n)+\frac{1}{2}\mathsf E\left(\varphi_{xx}(t_n,\widehat x_n)A\xi\cdot A\xi\right)\right.\\
  \left. -\varphi_t(\overline t,\overline x)-\frac{1}{2}\Tr\left(A A^T \varphi_{xx}(\overline t,\overline x)\right)
 \right|,
\end{gather*}
which yields the result by the dominated convergence theorem. \end{proof}

To apply Lemmas \ref{lem:2}, \ref{lem:3} in our case, put $m=d+1$, $U=Q=[0,1]\times\mathbb R^d$, $U_n=Q_n=\bigcup_{j=0}^n \{j/n\}\times\mathbb R^d$, and denote by $\underline v$, $\overline v$ the half-relaxed limits of $v_n$.

Let $(\overline t,\overline x)\in Q$ be the strict global minimum point of $\underline v-\varphi$ on $Q$ for a test function $\varphi\in C_b^2(\mathbb R^m)$. Take sequences $y_k=(t_{j(k)},x_k)\in Q_{n_k}\to (\overline t,\overline x)$, $v_{n_k}$ and a function $\psi$, given by Lemma \ref{lem:3}. If $\overline t<1$, we may assume that $t_{j(k)}<1$, that is, $j(k)<n$. By the recurrence relation (\ref{eq:3.3}) we have
$$ v_{n_k}(t_j(k),x_k)=\sup_{A\in\Lambda} \mathsf E v_{n_k}(t_{j(k)}+1/n_k,x_k+A\xi/\sqrt n_k).$$
The inequality
$ (v_{n_k}-\psi)(t,x)\ge (v_{n_k}-\psi)(t_{j(k)},x_k)$, $(t,x)\in Q_{n_k} $
implies that
$$ 0\ge \sup_{A\in\Lambda} \mathsf E \psi(t_{j(k)}+1/n_k,x_k+A\xi/\sqrt n_k)-\psi(t_j(k),x_k),$$
and Lemma \ref{lem:4} gives the inequality
\begin{equation} \label{eq:3.10}
 0 \le -\varphi_t(\overline t,\overline x)-G\left(\varphi_{xx}(\overline t,\overline x)\right),
\end{equation}
since the derivatives of $\varphi$ and $\psi$, up to the second order, coincide at $(\overline t,\overline x)$.

Now assume that $\overline t=1$. Clearly, $\underline v(1,x)\le f(x)$. If $\underline v(1,\overline x)<f(\overline x)$ then $t_{j(k)}<1$ for sufficiently large $k$ (since $v_{n_k}(1,x_k)=f(x_k)$ converges to $f(\overline x)$). So, we again obtain the inequality (\ref{eq:3.10}) as above. However, this is impossible. Indeed, we can change $\varphi$ to $\widehat\varphi=\varphi - c(1-t)$, $c>0$ in this inequality, since $(1,\overline x)$ is still the global minimum point of $\underline v-\widehat\varphi$:
$$ 0\le -c-\varphi_t-G\left(\varphi_{xx}(\overline t,\overline x)\right),\quad \text{for any } c>0.$$
This contradiction shows that $\underline v(1,x)=f(x)$, $x\in\mathbb R^d$.

Thus, we have proved that $\underline v$ is a viscosity supersolution of (\ref{eq:2.1}), (\ref{eq:2.2}). In the same way one can prove that $\overline v$ is a viscosity subsolution of (\ref{eq:2.1}), (\ref{eq:2.2}). By the comparison result, mentioned in Remark \ref{rem:1}, we have $\underline v\ge\overline v$. The opposite inequality follows from the definition of $\underline v$, $\overline v$. Therefore, $v=\underline v=\overline v$ is the unique continuous viscosity solution of (\ref{eq:2.1}), (\ref{eq:2.2}).



Finally, form the definition of $\underline v, \overline v$ we see that
$$ v(0,0)=\underline v(0,0)\le\liminf_{n\to\infty} v_n(0,0)\le\limsup_{n\to\infty} v_n(0,0)\le\overline v(0,0)=v(0,0).$$
In view of (\ref{eq:2.8}), this finishes the proof of Theorem 1.

\section{On the relationship with the sublinear expectations framework}
\label{sec:4}
Recall (see, e.g., \cite{Peng08c}) that a sublinear expectation space is a triple $(\Omega,\mathcal H,\widehat{\mathbb E})$, where $\Omega$ is a set, $\mathcal H$ is a linear space of real valued functions defined on $\Omega$, and $\widehat{\mathbb E}$ is a sublinear functional on $\mathcal H$. It is assumed that $\mathcal H$ contains constants and $|X|\in\mathcal H$ for $X\in\mathcal H$. Moreover, $\mathcal H$ should be invariant with respect to some functional transformations. The most standard assumption is the following:
$$ \varphi(X_1,\dots,X_n)\in\mathcal H\qquad \text{for }X_1,\dots,X_n\in\mathcal H,\quad \varphi\in C_{l,Lip}(\mathbb R^n),$$
where $C_{l,Lip}(\mathbb R^n)$ is the linear space of functions $\varphi$, satisfying the inequalities
$$ |\varphi(x)-\varphi(y)|\le C(1+|x|^m+|y|^m)|x-y|,\quad x,y\in\mathbb R^n,$$
with $C$, $m$, depending on $\varphi$.

Sublinear expectation $\widehat{\mathbb E}:\mathcal H\mapsto\mathbb R$ satisfies the following conditions:
\begin{itemize}
\item[(i)] $\widehat{\mathbb E} X\le \widehat{\mathbb E} Y$, $X\le Y$,
\item[(ii)] $\widehat{\mathbb E} c=c$ for a constant $c\in\mathbb R$,
\item[(iii)] $\widehat{\mathbb E}(X+Y)\le \widehat{\mathbb E} X+\widehat{\mathbb E} Y$,
\item[(iv)] $\widehat{\mathbb E} (\lambda X)=\lambda \widehat{\mathbb E} X$,\quad $\lambda\ge 0$.
\end{itemize}

To put our problem in this context, assume that the random variables $\xi_i$ have finite moments of all orders. Consider the space of sequences $\Omega=\{(x_i)_{i=1}^\infty:x_i\in\mathbb R^d\}$, and introduce the space of random variables $\mathcal H$ as follows: $\mathcal H=\cup_{n=1}^\infty \mathcal H_n$, where
$\mathcal H_n$ is the space of continuous functions $X=\psi(x_1,\dots,x_n)$ of polynomial growth. Define the sublinear expectation by the formula
$$ \widehat{\mathbb E}X=\sup_{A_0^{n-1}\in\mathfrak A_0^{n-1}}\mathsf E\psi(A_0\xi_1,\dots,A_{n-1}\xi_n).$$
It is easy to see that the triple $(\Omega,\mathcal H,\widehat{\mathbb E})$, defined in this way, is a sublinear expectation space.

Denote by $X_i$ the projection mappings: $X_i(x)=x_i$. For $\varphi\in C_{l,Lip}(\mathbb R^{d\times n})$ put
$$\overline\varphi(x_1,\dots,x_{n-1})=\widehat{\mathbb E}\varphi(x_1,\dots,x_{n-1},X_n)=\sup_{A\in\Lambda}\mathbb E\varphi(x_1,\dots,x_{n-1},A\xi_n).$$
Using the result of \cite{RockWets98} (Theorem 14.60) on the interchange of maximization and expectation operations, it is not difficult to show that
\begin{align*}
&\widehat{\mathbb E}\varphi(X_1,\dots,X_n)=\sup_{A_0^{n-1}\in\mathfrak A_0^{n-1}}\mathsf E\varphi(A_0\xi_1,\dots,A_{n-1}\xi_n)\\
&=\sup_{A_0^{n-2}\in\mathfrak A_0^{n-2}}\mathsf E\overline\varphi(A_0\xi_1,\dots,A_{n-2}\xi_{n-1})=\widehat{\mathbb E}\overline\varphi(X_1,\dots,X_{n-1}).
\end{align*}
This means that $X_n$ is independent from $(X_1,\dots,X_{n-1})$ in the sense of sublinear expectations theory (see \cite{Peng08c}, Definition 3.10).

Denote by $m_n(\Lambda)$ the set of $\mathscr F_n$-measurable functions with values in $\Lambda$. The random variables $X_i$ have no mean uncertainty:
$$ \widehat{\mathbb E}(\pm X_i)=\sup_{A_{i-1}\in m_{i-1}(\Lambda)}\mathsf E (\pm A_{i-1}\xi_i)=0.$$
Furthermore, for $S\in\mathbb S_d$ we obtain
\begin{align*}
\widehat{\mathbb E} (SX_i\cdot X_i) &=\sup_{A_{i-1}\in m_{i-1}(\Lambda)}\mathsf E(S A_{i-1}\xi_i\cdot A_{i-1}\xi_i)=
\sup_{A_{i-1}\in m_{i-1}(\Lambda)}\mathsf E\, \Tr(A_{i-1}^T S A_{i-1})\\
&=\sup_{A\in\Lambda}\Tr(A^T S A).
\end{align*}
By Peng's central limit theorem (see \cite[Theorem 5.1]{Peng08}, \cite[Theorem 3.3]{Peng08c}) the sequence $n^{-1/2}\sum_{i=0}^n X_i$ converges in law to a $G$-normal random vector $Y$:
$$ \widehat{\mathbb E} f\left(\frac{1}{\sqrt n}\sum_{i=1}^n X_i\right)=
\sup_{A_0^{n-1}\in\mathfrak A_0^{n-1}}\mathsf E f\left(\sum_{j=0}^{n-1}\frac{A_j\xi_{j+1}}{\sqrt n}\right)\to \widehat{\mathbb E} f(Y),\quad f\in C_b(\mathbb R)$$
with
$$ G(S)=\frac{1}{2}\widehat{\mathbb E} (SX_i\cdot X_i)=\frac{1}{2}\sup_{A\in\Lambda}\Tr(A^T S A).$$

By the definition of the $G$-normal distribution, we have $\widehat{\mathbb E} f(Y)=u(0,0)$, where $u$ is the viscosity solution of the $G$-heat equation (\ref{eq:2.1}) with the boundary condition (\ref{eq:2.2}). Hence, for $\xi_i$ with finite moments of all orders, Theorem 1 follows from Peng's CLT.

As is mentioned in \cite{Peng08c} (Remark 3.8) the same result can be proved under the assumption $\widehat{\mathbb E}|X_i|^{2+\delta}<\infty$ for some $\delta>0$, instead of $\widehat{\mathbb E}|X_i|^n<\infty$ for all $n$ (it is not known whether one can take $\delta=0$). In our case it corresponds to the condition $\mathbb E|\xi_i|^{2+\delta}<\infty$, which is still superfluous. The reason for the appearance of this assumption we see in the lack of dominated convergence theorem in a general sublinear expectation space. We also refer to \cite{Rokh15}, where Peng's approach was applied to the one-dimensional problem with variance uncertainty. It was shown that, written in the classical terms, this approach allows to prove the CLT without unnatural assumptions, even in the case of non-identically distributed independent random variables.

Conversely, one can try to prove Peng's CLT by the methods of the present paper. Let $(X_i)_{i=1}^\infty$ be i.i.d. random variables under a sublinear expectation $\widehat{\mathbb E}$, and let $f$ be a bounded Lipshitz continuous function. The definition of independence, applied to bounded Lipshitz continuous functions
$$ v_n(t_j^n,x)=\widehat{\mathbb E} f\left(x+\frac{1}{n^{1/2}}\sum_{i=j}^n X_i\right),\quad t_j^n=j/n,$$
gives the recurrence relation
$$ v_n(t_j^n,x)=\widehat{\mathbb E}\left[\widehat{\mathbb E}f\left(x+\frac{x_j}{n^{1/2}}+\frac{1}{n^{1/2}}\sum_{i=j+1}^n X_i\right)\biggr|_{x_j=X_j}\right]=\widehat{\mathbb E}v_n(t_{j+1}^n,x+X_j/n^{1/2})$$
corresponding to the dynamic programming principle, considered in Lemma \ref{lem:1}. However, to follow subsequent reasoning, one needs additional assumptions, like monotone continuity (or Fatou) property: see \cite{CohJiPen11}. Indeed, we have used the dominated convergence theorem which is not true in a general sublinear expectation space.

 \bibliographystyle{plain}
 \bibliography{litCLT}

\begin{thebibliography}{10}

\bibitem{AveLevPar95}
M.~Avellaneda, A.~Levy, and A.~Par{\'a}s.
\newblock Pricing and hedging derivative securities in markets with uncertain
  volatilities.
\newblock {\em Appl. Math. Finance}, 2(2):73--88, 1995.

\bibitem{BarCap97}
M.~Bardi and I.~Capuzzo-Dolcetta.
\newblock {\em Optimal control and viscosity solutions of
  {H}amilton-{J}acobi-{B}ellman equations}.
\newblock Birkhauser, Boston, 1997.

\bibitem{BarPer88}
G.~Barles and B.~Perthame.
\newblock Exit time problems in optimal control and vanishing viscosity method.
\newblock {\em SIAM J. Control Optim.}, 26(5):1133--1148, 1988.

\bibitem{BarSou91}
G.~Barles and P.~E. Souganidis.
\newblock Convergence of approximation schemes for fully nonlinear second order
  equations.
\newblock {\em Asymptot. Anal.}, 4:271--283, 1991.

\bibitem{Bass11}
R.F. Bass.
\newblock {\em Stochastic processes}.
\newblock Cambridge University Press, Cambridge, 2011.

\bibitem{CheGigGot91}
Y.-G. Chen, Y.~Giga, and S.~Goto.
\newblock Remarks on viscosity solutions for evolution equations.
\newblock {\em Proc. Japan Acad., Ser. A}, 67(10):323--328, 1991.

\bibitem{CohJiPen11}
S.~Cohen, S.~Ji, and S.~Peng.
\newblock Sublinear expectations and martingales in discrete time.
\newblock {P}reprint arXiv:1104.5390 [math.PR], 26 pages, 2011.

\bibitem{CraIshLio92}
M.~Crandall, H.~Ishii, and P.-L. Lions.
\newblock User's guide to viscosity solutions of second-order partial
  differential equations.
\newblock {\em Bull. Amer. Math. Soc.}, 27:1--67, 1992.

\bibitem{DieFriObe14}
J.~Diehl, P.K. Friz, and H.~Oberhauser.
\newblock Regularity theory for rough partial differential equations and
  parabolic comparison revisited.
\newblock In {\em Stochastic Analysis and Applications 2014}, pages 203--238.
  Springer, 2014.

\bibitem{Nis15}
M.~Nisio.
\newblock {\em Stochastic control theory: dynamic programming principle}.
\newblock Probability Theory and Stochastic Modelling, 72. Springer, Tokyo,
  2015.

\bibitem{Peng07b}
S.~Peng.
\newblock {$G$}-expectation, {$G$}-{B}rownian motion and related stochastic
  calculus of {I}t{\^o} type.
\newblock In {\em Stochastic analysis and applications: The Abel Symposium
  2005}, pages 541--567. Springer, 2007.

\bibitem{Peng07}
S.~Peng.
\newblock Law of large numbers and central limit theorem under nonlinear
  expectations.
\newblock {P}reprint arXiv:math/0702358 [math.PR], 8 pages, 2007.

\bibitem{Peng08b}
S.~Peng.
\newblock Multi-dimensional {$G$}-{B}rownian motion and related stochastic
  calculus under {$G$}-expectation.
\newblock {\em Stoch. Proc. Appl.}, 118(12):2223--2253, 2008.

\bibitem{Peng08}
S.~Peng.
\newblock A new central limit theorem under sublinear expectations.
\newblock {P}reprint arXiv:0803.2656 [math.PR], 25 pages, 2008.

\bibitem{Peng08c}
S.~Peng.
\newblock Nonlinear expectations and stochastic calculus under uncertainty.
\newblock {P}reprint arXiv:1002.4546 [math.PR], 149 pages, 2010.

\bibitem{RockWets98}
R.T. Rockafellar and R.J.-B. Wets.
\newblock {\em Variational analysis}.
\newblock Springer-Verlag, Berlin, 2009.

\bibitem{Rokh15}
D.B. Rokhlin.
\newblock Central limit theorem under variance uncertainty.
\newblock {P}reprint arXiv:1506.01551 [math.PR], 9 pages, 2015.

\bibitem{Schal74}
M.~Sch{\"a}l.
\newblock A selection theorem for optimization problems.
\newblock {\em Arch. Math.}, 25(1):219--224, 1974.

\bibitem{Sir14}
M.~S{\^\i}rbu.
\newblock A note on the strong formulation of stochastic control problems with
  model uncertainty.
\newblock {\em Electron. Comm. Probab.}, 19:1--10, 2014.

\bibitem{Sri98}
S.M. Srivastava.
\newblock {\em A course on {B}orel sets}.
\newblock Springer-Verlag, New York, 1998.

\bibitem{Tou13}
N.~Touzi.
\newblock {\em Optimal stochastic control, stochastic target problems, and
  backward {S}{D}{E}}.
\newblock Fields Institute Monographs, 29. Springer, New York, 2013.

\end{thebibliography}





\end{document}